\documentclass[preprint,12pt]{elsarticle}

\usepackage{amssymb}
 \usepackage{graphicx}
\usepackage{epsfig}

\usepackage{amsthm}

\newcommand{\sysn}{\left\{\begin{array}{rcl}}
\newcommand{\sysk}{\end{array}\right.}

\newcommand{\ingrw}[2]{\includegraphics[width=#1mm]{#2}}

\newtheorem{theorem}{Theorem}[section]

\theoremstyle{example}

\newtheorem{proposition}[theorem]{Proposition}
\theoremstyle{definition}
\newtheorem{definition}[theorem]{Definition}
\newtheorem{remark}[theorem]{Remark}
\newtheorem{corollary}[theorem]{Corollary}

\journal{Topology and its Applications}

\begin{document}

\begin{frontmatter}



\title{Indestructibly productively Lindel$\ddot{o}$f and Menger function spaces}


\author{Alexander V. Osipov}

\ead{OAB@list.ru}


\address{Krasovskii Institute of Mathematics and Mechanics, Ural Federal
 University,

 Ural State University of Economics, Yekaterinburg, Russia}

\begin{abstract}
For a Tychonoff space $X$ and a family $\lambda$ of subsets of
$X$, we denote by $C_{\lambda}(X)$ the $T_1$-space of all
real-valued continuous functions on $X$ with the $\lambda$-open
topology.

A topological space is productively Lindel$\ddot{o}$f if its
product with every Lindel$\ddot{o}$f space is Lindel$\ddot{o}$f. A
space is indestructibly productively Lindel$\ddot{o}$f if it is
productively Lindel$\ddot{o}$f in any extension by countably
closed forcing. A Menger space is a topological space in which for
every sequence of open covers $\mathcal{U}_1, \mathcal{U}_2,...$
of the space there are finite sets $\mathcal{F}_1\subset
\mathcal{U}_1, \mathcal{F}_2\subset \mathcal{U}_2, ...$ such that
family $\mathcal{F}_1\cup \mathcal{F}_2\cup ...$ covers the space.

In this paper, we study indestructibly productively
Lindel$\ddot{o}$f and Menger function spaces. In particular, we
proved that the following statements are equivalent for a
$T_1$-space $C_{\lambda}(X)$:

(1) $C_{\lambda}(X)$ is indestructibly productively
Lindel$\ddot{o}$f;

(2) $C_{\lambda}(X)$ is metrizable Menger;

(3) $C_{\lambda}(X)$ is metrizable $\sigma$-compact;

(4)  $X$ is pseudocompact, $D(X)$ is a dense $C^*$-embedded set in
$X$ and the family $\lambda$ consists of all finite subsets of
$D(X)$, where $D(X)$ is the countable set of all isolated points
of $X$;

(5) $C_{\lambda}(X)$ is homeomorphic to $C_p^*(\mathbb{N})$.

\end{abstract}

\begin{keyword} Menger property \sep Hurewicz property \sep  set-open topology \sep
$\sigma$-compact \sep  function space \sep indestructibly
Lindel$\ddot{o}$f  \sep productively Lindel$\ddot{o}$f \sep
indestructibly productively Lindel$\ddot{o}$f \sep selection
principles


\MSC[2010]  54C25 \sep 54C35 \sep 54C40

\end{keyword}

\end{frontmatter}



\section{Introduction}
A space $X$ is said to be Menger \cite{hur} (or, \cite{sash}) ($X$
satisfies $S_{fin}(\mathcal{O}, \mathcal{O})$) if for every
sequence $(\mathcal{U}_n : n\in \mathbb{N})$ of open covers of
$X$, there are finite subfamilies $\mathcal{V}_n\subset
\mathcal{U}_n$ such that $\bigcup \{\mathcal{V}_n : n\in
\mathbb{N} \}$ is a cover of $X$.

Note that every $\sigma$-compact space is Menger, and a Menger
space is Lindel$\ddot{o}$f. The Menger property is closed
hereditary, and it is preserved by continuous maps. It is well
known that the Baire space $\mathbb{N}^{\mathbb{N}}$ (hence,
$\mathbb{R}^{\omega}$) is not Menger.

Menger conjectured that in {\bf ZFC} every Menger metric space is
$\sigma$-compact. Fremlin and Miller \cite{mi} proved that
Menger's conjecture is false, by showing that there is, in {\bf
ZFC}, a set of real numbers that is Menger but not
$\sigma$-compact.

For a function space $C_p(X)$: Velichko proved that $C_p(X)$ is
$\sigma$-compact iff $X$ is finite and Arhangel'skii proved that
$C_p(X)$ is Menger iff $X$ is finite \cite{arh21}.

For a function space $C_{\lambda}(X)$  the situation is more
interesting.

\begin{theorem}(Nokhrin)\label{nox}(Theorem 5.13 in \cite{nokh})
 For a space $X$ the following statements are
equivalent:

\begin{enumerate}
\item $C_{\lambda}(X)$ is $\sigma$-compact;

\item $X$ is pseudocompact, $D(X)$ is a dense $C^*$-embedded set
in $X$ and the family $\lambda$ consists of all finite subsets of
$D(X)$, where $D(X)$ is the set of all isolated points of $X$.
\end{enumerate}

\end{theorem}

\begin{theorem}(Osipov)\label{os1} (Theorem 3.4 in \cite{os3})
 A space $C_{\lambda}(X)$ is Menger iff  it is
$\sigma$-compact.
\end{theorem}

 Various properties between $\sigma$-compactness and the Menger property are
 investigated in the papers \cite{dtz,tall1,tall} and others. We continue to study these
 properties on a function $T_1$-space $C(X)$ with the set-open topology.

\medskip

\section{Main definitions and notation}

Throughout this paper $X$ will be a Tychonoff space.  Let
$\lambda$ be a family nonempty subsets of $X$ and let $C(X)$ be a
set of all continuous real-valued functions on $X$. Denote by
$C_{\lambda}(X)$  the set $C(X)$ is endowed with the
$\lambda$-open topology.
 The elements of the standard subbases of the set-open
 topology:

 $[F,\,U]=\{f\in C(X):\ f(F)\subseteq U\}$,  where $F\in\lambda$, $U$ is an open subset  of the real line $\mathbb
 R$.

 Note that if $\lambda$ consists of all
finite (compact) subsets of $X$ then the $\lambda$-open topology
coincides with the topology of pointwise convergence (the
compact-open topology), that is $C_{\lambda}(X)=C_{p}(X)$
($C_{\lambda}(X)=C_{k}(X)$). The set-open topology was first
introduced by Arens and Dugundji in \cite{ardu} and studied over
the last years by many authors. We continue to study the different
topological properties of the space $C(X)$ with the set-open
topology (see [21-27]).


For a topological property $\mathcal{P}$, A.V. Arhangel'skii calls
$X$ {\it projectively $\mathcal{P}$} if every second countable
image of $X$ is $\mathcal{P}$. Arhangel'skii consider projective
$\mathcal{P}$ for $\mathcal{P}=\sigma$-compact, analytic and other
properties \cite{arh}. The projective selection principles were
introduced and first time considered in \cite{koc}. Lj.D.R.
Ko$\check{c}$inac characterized the classical covering properties
of Menger, Rothberger, Hurewicz and Gerlits-Nagy in term of
continuous images in $\mathbb{R}^{\omega}$.

\medskip

\begin{theorem}(Ko$\check{c}$inac)\label{koc} A space is Menger if and only if it is Lindel$\ddot{o}$f
and projectively Menger.
\end{theorem}

Recall that, if $X$ is a topological space and $\mathcal{P}$ is a
topological property, we say that $X$ is $\sigma$-$\mathcal{P}$ if
$X$ is the countable union of subspaces with the property
$\mathcal{P}$. So a space $X$ is called $\sigma$-compact
($\sigma$-pseudocompact, $\sigma$-bounded), if
$X=\bigcup\limits_{i=1}^{\infty}X_i$, where $X_i$ is a compact
(pseudocompact, bounded) for every $i\in \mathbb N$.

 A subset $A$ of a space $X$ is said to be {\it
 bounded} in $X$ if for every continuous function $f: X \mapsto
 \mathbb{R}$, $f\vert A: A\mapsto \mathbb{R}$ is a bounded
 function. Every $\sigma$-bounded space is projectively Menger
 (Proposition 1.1 in \cite{arh}).

\medskip

Recall that a family $\lambda$ of nonempty subsets of a
topological space $(X,\tau)$ is called a $\pi$-network for $X$ if
for any nonempty open set $U\in\tau$ there exists $A\in \lambda$
such that $A\subset U$.

By Theorem 4.1 in \cite{nokh}, the space $C_{\lambda}(X)$ is a
$T_1$-space (=Hausdorff space) iff $\lambda$ is a $\pi$-network of
$X$.

Throughout this paper, a family $\lambda$ of nonempty subsets of
the set $X$ is
  a $\pi$-network.

Many topological properties are defined or characterized in terms
 of the following classical selection principles.
 Let $\mathcal{A}$ and $\mathcal{B}$ be sets consisting of
families of subsets of an infinite set $X$. Then:

$S_{1}(\mathcal{A},\mathcal{B})$ is the selection hypothesis: for
each sequence $( A_{n}: n\in \mathbb{N})$ of elements of
$\mathcal{A}$ there is a sequence $( b_{n}: n\in \mathbb{N})$ such
that for each $n$, $b_{n}\in A_{n}$, and $\{b_{n}: n\in\mathbb{N}
\}$ is an element of $\mathcal{B}$.

$S_{fin}(\mathcal{A},\mathcal{B})$ is the selection hypothesis:
for each sequence $( A_{n}: n\in \mathbb{N} )$ of elements of
$\mathcal{A}$ there is a sequence $( B_{n}: n\in \mathbb{N})$ of
finite sets such that for each $n$, $B_{n}\subseteq A_{n}$, and
$\bigcup_{n\in\mathbb{N}}B_{n}\in\mathcal{B}$.

$U_{fin}(\mathcal{A},\mathcal{B})$ is the selection hypothesis:
whenever $\mathcal{U}_1$, $\mathcal{U}_2, ... \in \mathcal{A}$ and
none contains a finite subcover, there are finite sets
$\mathcal{F}_n\subseteq \mathcal{U}_n$, $n\in \mathbb{N}$, such
that $\{\bigcup \mathcal{F}_n : n\in \mathbb{N}\}\in \mathcal{B}$.

In this paper, by a cover we mean a nontrivial one, that is,
$\mathcal{U}$ is a cover of $X$ if $X=\bigcup \mathcal{U}$ and
$X\notin \mathcal{U}$.

 An open cover $\mathcal{U}$ of a space $X$ is:

 $\bullet$ an {\it $\omega$-cover} if every finite subset of $X$ is contained in a
 member of $\mathcal{U}$;

$\bullet$ a {\it $\gamma$-cover} if it is infinite and each $x\in
X$ belongs to all but finitely many elements of $\mathcal{U}$.

For a topological space $X$ we denote:

$\bullet$ $\mathcal{O}$ --- the family of open covers of $X$;

$\bullet$ $\Gamma$ --- the family of open $\gamma$-covers of $X$;

$\bullet$ $\Omega$ --- the family of open $\omega$-covers of $X$.

Many equivalences hold among these properties, and the surviving
ones appear in the following Diagram (where an arrow denotes
implication), to which no arrow can be added except perhaps from
$U_{fin}(\Gamma, \Gamma)$ or $U_{fin}(\Gamma, \Omega)$ to
$S_{fin}(\Gamma, \Omega)$ \cite{jmss}.

\bigskip

\begin{center}
\ingrw{90}{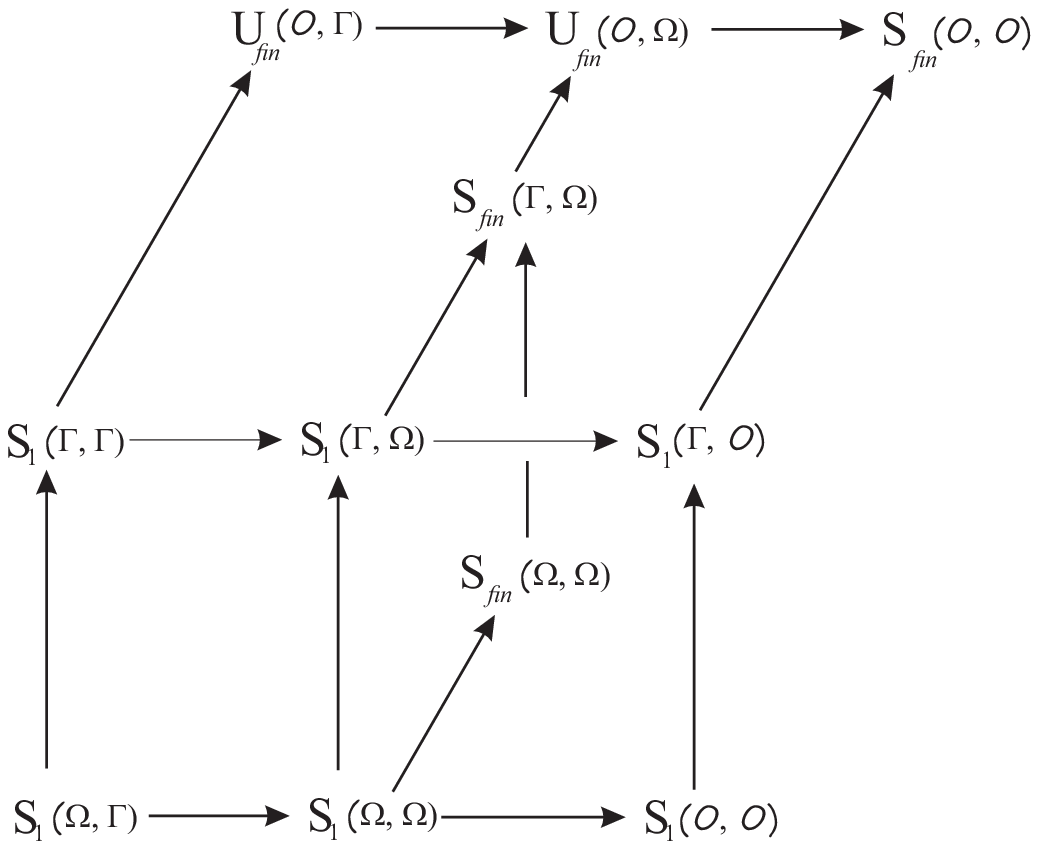}

\medskip

Figure~1. The Scheepers Diagram for Lindel$\ddot{o}$f spaces.

\end{center}
\bigskip

\begin{definition} A topological space $X$ is

$\bullet$ {\it Indestructibly Lindel$\ddot{o}$f} if it is
Lindel$\ddot{o}$f in every countably closed forcing extension
\cite{tall4,auta}.

$\bullet$ {\it Productively Lindel$\ddot{o}$f} if $X\times Y$ is
Lindel$\ddot{o}f$ for any Lindel$\ddot{o}$f space $Y$ \cite{bkr}.

$\bullet$ {\it Indestructibly productively Lindel$\ddot{o}$f} if
it is productively Lindel$\ddot{o}$f in any extension by countably
closed forcing \cite{auta}.

$\bullet$ {\it Powerfully Lindel$\ddot{o}$f} if its $\omega$th
power is Lindel$\ddot{o}$f \cite{al,auta}.

$\bullet$ {\it Alster} if every cover by $G_{\delta}$ sets that
covers each compact set finitely includes a countable subcover
\cite{al}.

$\bullet$ {\it Rothberger} ($X$ satisfies $S_1(\mathcal{O},
\mathcal{O})$) if for each sequence of open covers
$\{\mathcal{U}_n\}_{n<\omega}$, there are $\{U_n\}_{n<\omega}$,
$U_n\in \mathcal{U}_n$, such that $\{U_n\}_{n<\omega}$ is a cover
\cite{rot}.

$\bullet$ {\it Hurewicz} ($X$ satisfies $U_{fin}(\mathcal{O},
\Gamma)$) if for each sequence $\{\mathcal{U}_n : n<\omega \}$ of
$\gamma$-covers, there is for each $n$ a finite
$\mathcal{V}_n\subset\mathcal{U}_n$ such that either $\{\bigcup
\mathcal{V}_n : n<\omega \}$ is a $\gamma$-cover, or else for some
$n$, $\mathcal{V}_n$ is a cover \cite{hur,hur1}.

\end{definition}

\begin{definition}
We play a {\it Menger game} ($M$-game) in which ONE chooses in the
$n$th inning an open cover $\mathcal{U}_n$ and TWO choses a finite
$\mathcal{V}_n\subset\mathcal{U}_n$. TWO wins if $\{\bigcup
\mathcal{V}_n : n<\omega\}$ covers $X$.
\end{definition}

Hurewicz proved $X$ is Menger if and only if ONE has no winning
strategy \cite{hur}.

\medskip

The Tall's Diagram in Figure 2 (see Diagram  in \cite{tall}) below
shows the relationships among the properties we have discussed in
this article.

\begin{center}
\ingrw{90}{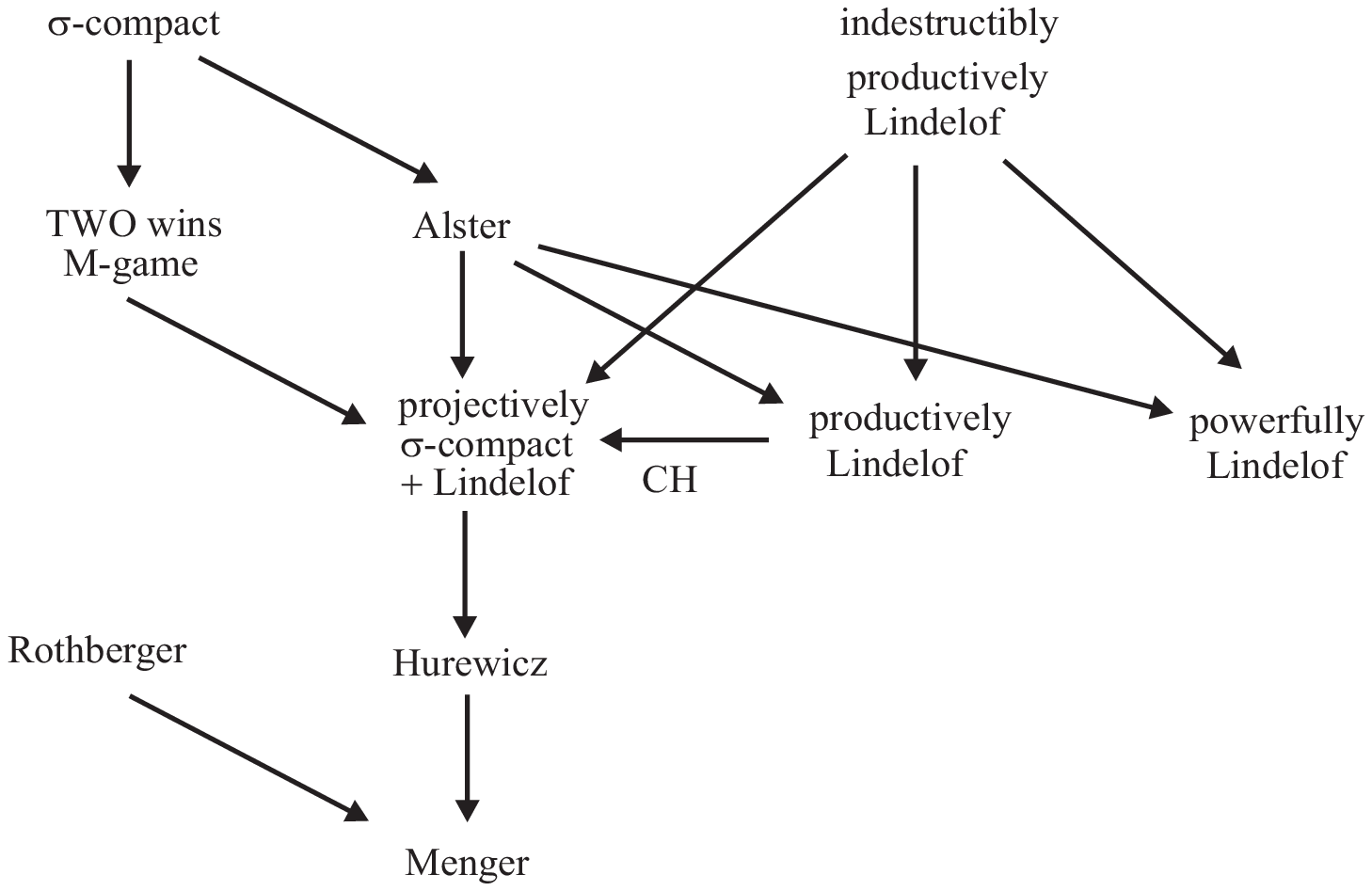}

\medskip

Figure~2. The Tall's Diagram for Lindel$\ddot{o}$f spaces.

\end{center}

\begin{theorem}\label{os2}

For a space $X$ the following statements are equivalent:

\begin{enumerate}

\item $C_{\lambda}(X)$ is $\sigma$-compact;

\item $C_{\lambda}(X)$ is Alster;

\item $(CH)$ $C_{\lambda}(X)$ is productively Lindel$\ddot{o}$f;

\item "TWO wins $M$-game" for $C_{\lambda}(X)$;

\item $C_{\lambda}(X)$ is projectively $\sigma$-compact and
Lindel$\ddot{o}$f;

\item $C_{\lambda}(X)$ is Hurewicz;

\item  $C_{\lambda}(X)$ is Menger;

\item $X$ is a pseudocompact, $D(X)$ is a dense $C^*$-embedded set
in $X$ and family $\lambda$ consists of all finite subsets of
$D(X)$, where $D(X)$ is the set of all isolated points of $X$;

\item  $(C_{\lambda}(X))^n$ is Menger for every $n\in \mathbb{N}$;

\item $C_{\lambda}(X)$ satisfies $S_{fin}(\Omega,\Omega)$;

\item $C_{\lambda}(X)$ is $\sigma$-countably compact and
Lindel$\ddot{o}$f;

\item $C_{\lambda}(X)$ is $\sigma$-pseudocompact and
Lindel$\ddot{o}$f;

\item $C_{\lambda}(X)$ is $\sigma$-bounded and Lindel$\ddot{o}$f.

\item $C_{\lambda}(X)$ is homeomorphic to
$\bigcup\limits_{i=1}^{\infty} [-i,i]^{|D(X)|}$.

\end{enumerate}

\end{theorem}

\begin{proof}  $(1)\Rightarrow(2)$. It is obvious that every
$\sigma$-compact space is Alster.

$(1)\Rightarrow(4)$. It is obvious that every $\sigma$-compact
space has "TWO wins $M$-game".

$(2)\Rightarrow(3)$. Every Alster space is productively
Lindel$\ddot{o}$f \cite{al}.

$(2)\Rightarrow(5)$. Since Alster metrizable spaces are
$\sigma$-compact \cite{al}, then every Alster space is
projectively $\sigma$-compact and Lindel$\ddot{o}$f.

$(4)\Rightarrow(5)$. By Theorem 18 in \cite{tall}.

$(7)\Rightarrow(1)$. By Theorem \ref{os1}.

$(9)\Leftrightarrow(10)$. By Theorem 3.9 in \cite{jmss}.

 $(13)\Rightarrow(1)$. Every $\sigma$-bounded space is projectively Menger
 (Proposition 1.1 in \cite{arh}). By Theorem \ref{koc},
 $C_{\lambda}(X)$ is Menger. By Theorem \ref{os1}, $C_{\lambda}(X)$
 is $\sigma$-compact.

$(1)\Leftrightarrow(14)$. By Theorem 5.5 in \cite{nokh}.

The remaining implications are trivial and follows from the
definitions \cite{tall}.
\end{proof}

\begin{corollary} If $C_{\lambda}(X)$ is Menger. Then
$C_{\lambda}(X)$ is powerfully (productively) Lindel$\ddot{o}$f.
\end{corollary}

\begin{proof} If $C_{\lambda}(X)$ is Menger, then, by Theorem
\ref{os2}, $C_{\lambda}(X)$ is $\sigma$-compact and hence it is
Alster. But Alster spaces are powerfully (productively)
Lindel$\ddot{o}$f \cite{al}.
\end{proof}

\begin{corollary} $C_{k}(X)$ is Menger ($\sigma$-compact, Hurewicz, Alster)  if and only if $X$ is finite.
\end{corollary}

For the selection properties of the space $C_{\lambda}(X)$  (see
Fig. 2) we have next trivial corollaries.

\begin{corollary}  If $C_{\lambda}(X)$ is Rothberger (in particular, $S_1(\Omega,
\Gamma)$ or $S_1(\Omega, \Omega)$). Then $X=\emptyset$.
\end{corollary}

\begin{proof} If $C_{\lambda}(X)$ is Rothberger, then it is Menger
and, by Theorem \ref{os2}, $X$ contains an isolated point. Hence,
the real line $\mathbb{R}\subset C_{\lambda}(X)$. But every
Rothberger subset of the real line has strongly measure zero
\cite{rot}. It follows that $X=\emptyset$.
\end{proof}

\begin{corollary}  If $C_{\lambda}(X)$ has the
property $S_1(\Gamma,\mathcal{O})$ (in particular, $S_1(\Gamma,
\Gamma)$ or $S_1(\Gamma, \Omega)$). Then $X=\emptyset$.
\end{corollary}

\begin{proof} If $C_{\lambda}(X)$ has the
property $S_1(\Gamma,\mathcal{O})$, then it is Menger and, by
Theorem \ref{os2}, $X$ contains an isolated point. Hence, the
Cantor set $2^{\omega} \subset C_{\lambda}(X)$. But the Cantor
set, $2^{\omega}$, is not in the class $S_1(\Gamma,\mathcal{O})$
\cite{jmss}. It follows that $X=\emptyset$.
\end{proof}

Note that every $\sigma$-compact topological space is a member of
both the class $S_{fin}(\Omega,\Omega)$ and
$U_{fin}(\Gamma,\Gamma)$ (Theorem 2.2 in \cite{jmss}). It follows
that if $C_{\lambda}(X)$ is Menger then $C_{\lambda}(X)$ has the
properties $S_{fin}(\Omega,\Omega)$ and
$U_{fin}(\mathcal{O},\Gamma)$ (in particular,
$S_{fin}(\Gamma,\Omega)$ and $U_{fin}(\mathcal{O}, \Omega)$).

\begin{remark}
If $X$ is compact and $C_{\lambda}(X)$ is Menger, then
 $X$ is homeomorphic to $\beta(D)$, where $\beta(D)$ is  Stone-$\check{C}$ech compactification
of a discrete space $D$, and  $\lambda=[D]^{<\aleph_0}$.
\end{remark}

A Lindel$\ddot{o}$f space $Y$ is called a {\it Michael space}, if
$\omega^{\omega}\times Y$ is not Lindel$\ddot{o}$f.

\medskip

$\bullet$ (Repovs-Zdomskyy) If there exists a Michael space (this
follows from $\mathfrak{b}=\aleph_1$ or $\mathfrak{d} = Cov(M)$),
then every productively Lindel$\ddot{o}$f spaces has the Menger
property (Proposition 3.1 in \cite{rezd}).

$\bullet$ (Repovs-Zdomskyy) If $Add(\mathcal{M}) =\mathfrak{d}$,
then every productively Lindel$\ddot{o}$f space has the Hurewicz
property (Theorem 1.1 in \cite{rezd2}).

$\bullet$ (Zdomskyy) If $\mathfrak{u}=\aleph_1$, then every
productively Lindel$\ddot{o}$f space has the Hurewicz property.

$\bullet$ (Tall) $\mathfrak{d}=\aleph_1$ implies productively
Lindel$\ddot{o}$f spaces are Hurewicz. (Theorem 10 in
\cite{tall}).

$\bullet$ (Tall) $\mathfrak{b}=\aleph_1$ implies every
productively Lindel$\ddot{o}$of space is Menger (Theorem 7 in
\cite{tall}).

\begin{proposition} If $\mathfrak{b}=\aleph_1$ (or $Add(\mathcal{M})
=\mathfrak{d}$ or $\mathfrak{u}=\aleph_1$ or
$\mathfrak{d}=\aleph_1$), then every productively
Lindel$\ddot{o}$f
 space $C_{\lambda}(X)$ is $\sigma$-compact.
\end{proposition}

\begin{proof} If $\mathfrak{b}=\aleph_1$ (or $Add(\mathcal{M})
=\mathfrak{d}$ or $\mathfrak{u}=\aleph_1$ or
$\mathfrak{d}=\aleph_1$) and $C_{\lambda}(X)$ is productively
Lindel$\ddot{o}$f, then $C_{\lambda}(X)$ is Menger. By Theorem
\ref{os1}, $C_{\lambda}(X)$ is $\sigma$-compact.
\end{proof}

Denote by $C^*_p(\mathbb{N})$ the set of all bounded continuous
real-valued functions on $\mathbb{N}$ with the topology of
pointwise convergence.

\begin{theorem}
For a space $X$ the following statements are equivalent:

\begin{enumerate}

\item $C_{\lambda}(X)$ is indestructibly productively
Lindel$\ddot{o}$f;

\item $C_{\lambda}(X)$ is metrizable $\sigma$-compact;

\item $C_{\lambda}(X)$ is metrizable Menger;

\item $X$ is a pseudocompact, $D(X)$ is a dense $C^*$-embedded set
in $X$, family $\lambda$ consists of all finite subsets of $D(X)$,
where $D(X)$ is the countable set of all isolated points of $X$;

\item $C_{\lambda}(X)$ is homeomorphic to $C^*_p(\mathbb{N})$.

\end{enumerate}

\end{theorem}

\begin{proof} $(1)\Rightarrow(4)$. By Theorem 9 in \cite{tall}, indestructibly productively Lindel$\ddot{o}$f spaces are projectively
$\sigma$-compact and hence Hurewicz and Menger. By Theorems
\ref{nox} and \ref{os1},  $X$ is a pseudocompact, $D(X)$ is a
dense $C^*$-embedded set in $X$ and family $\lambda$ consists of
all finite subsets of $D(X)$, where $D(X)$ is the set of all
isolated points of $X$.

Assume that $\kappa=|D(X)|>\aleph_0$. Then $C_{\lambda}(X,
\mathbb{I})$ is homeomorphic to the space $\mathbb{I}^{\kappa}$
(Theorem 5.5 in \cite{nokh}) where $\mathbb{I}=[-1,1]$. It follows
that the compact space $C_{\lambda}(X, \mathbb{I})$ includes a
copy of $2^{\omega_1}$. Note that every indestructibly
productively Lindel$\ddot{o}$f space is indestructibly
Lindel$\ddot{o}$f. By Lemma 4.7 and Corollary 4.4 in \cite{tall3},
$C_{\lambda}(X)$ is destructibility. It follows that
$|D(X)|\leq\aleph_0$.

$(4)\Rightarrow(3)$. By Theorem \ref{os2}, $C_{\lambda}(X)$ is
Menger. Clearly that, if $|D(X)|\leq \aleph_0$ then
$w(C_{\lambda}(X))= \aleph_0$ and hence $C_{\lambda}(X)$ is
metrizable.

$(3)\Rightarrow(2)$. By Theorem \ref{os1}.

$(2)\Rightarrow(1)$. By Theorem 7 in \cite{auta}, a metrizable
space is indestructibly productively Lindel$\ddot{o}$f if and only
if it is $\sigma$-compact.

$(5)\Leftrightarrow(2)$. By Theorem 5.5 in \cite{nokh}, if
$C_{\lambda}(X)$ is $\sigma$-compact then  $C_{\lambda}(X)$ is
homeomorphic to the space $\bigcup\limits_{i=1}^{\infty}
[-i,i]^{\kappa}$ where $\kappa=|D(X)|$. Since $C_{\lambda}(X)$ is
metrizable then $\kappa\leq \aleph_0$. It follows that
$C_{\lambda}(X)$ is homeomorphic to $\bigcup\limits_{i=1}^{\infty}
[-i,i]^{\aleph_0}$. Note that $\bigcup\limits_{i=1}^{\infty}
[-i,i]^{\aleph_0}$ is homeomorphic to $C^*_p(\mathbb{N})$.

\end{proof}

\begin{corollary} $C_k(X)$ is indestructibly productively
Lindel$\ddot{o}$f iff $X$ is finite.

\end{corollary}

\begin{remark}
If $X$ is compact and $C_{\lambda}(X)$ is indestructibly
productively Lindel$\ddot{o}$f, then
 $X$ is homeomorphic to $\beta \mathbb{N}$, where $\beta \mathbb{N}$ is  Stone-$\check{C}$ech compactification
of the natural numbers $\mathbb{N}$, and
$\lambda=[\mathbb{N}]^{<\aleph_0}$.
\end{remark}

\bibliographystyle{model1a-num-names}
\bibliography{<your-bib-database>}







\end{document}